\documentclass{amsart} 
\usepackage{amsthm}
\usepackage{amsmath, ifthen, amssymb}
\usepackage{latexsym}
\usepackage{amsfonts}
\usepackage{bm}

\newcommand{\om}{\ensuremath{\omega}}
\newcommand{\set}[2]{\ensuremath{\{#1 \hspace{0.3mm} \mid \hspace{0.3mm} #2\}}}
\newcommand{\del}{\ensuremath{\Delta^1_1}}
\newcommand{\ca}[1]{\ensuremath{\mathcal{#1}}}
\newcommand{\pii}{\ensuremath{\Pi^1_1}}
\newcommand{\sig}{\ensuremath{\Sigma^1_1}}
\newcommand\tboldsymbol[1]{%
\protect\raisebox{0pt}[0pt][0pt]{%
$\underset{\widetilde{}}{\boldsymbol{#1}}$}\mbox{\hskip 1pt}}
\newcommand{\ep}{\ensuremath{\varepsilon}}

\newcommand{\lh}{{\rm lh}}
\newcommand{\ck}{\ensuremath{\om_1^{CK}}}
\newcommand{\ckr}[1]{\ensuremath{\om_1^{#1}}}
\newcommand{\hleq}{\ensuremath{\leq_{\rm h}}}
\newcommand{\heq}{\ensuremath{=_{\rm h}}}
\newcommand{\hless}{\ensuremath{<_{\rm h}}}

\newcommand{\rfn}[1]{\ensuremath{\{#1\}}}

\newcommand{\tu}[1]{\textup{#1}}

\newcommand{\bolds}{\ensuremath{\tboldsymbol{\Sigma}}}
\newcommand{\boldp}{\ensuremath{\tboldsymbol{\Pi}}}

\newcommand{\cn}[2]{\ensuremath{#1 \ast #2}}

\newcommand{\baire}{\ensuremath{{\mathcal{N}}}}
\newcommand{\cantor}{{\ensuremath{^\om2}}}

\newcommand{\pr}{{\rm pr}}

\newcommand{\pair}[1]{\langle #1 \rangle}

\newcommand{\myemph}[1]{{\bf \emph{#1}}}
\newcommand{\fix}{\text{FIX}}
\newcommand{\HYP}{\textit{HYP}}
\renewcommand{\bold}[1]{\bm{#1}}
\newcommand{\ie}{\text{i.e.,}~}
\newcommand{\hcode}{I}
\newcommand{\hsett}{{\rm Hp}}
\newcommand{\hset}[1]{\hsett(#1)}
\newcommand{\normI}[1]{|#1|_{\hcode}}
\newcommand{\KP}{\textsf{KP}}

\newtheorem{theorem}{Theorem}
\newtheorem{lemma}[theorem]{Lemma}
\newtheorem{definition}[theorem]{Definition}
\newtheorem{proposition}[theorem]{Proposition}
\newtheorem{corollary}[theorem]{Corollary}
\newtheorem{question}[theorem]{Question}
\newtheorem{remark}[theorem]{Remark}

\newtheorem{claim}[theorem]{Claim}
\newtheorem{conjecture}[theorem]{Conjecture}

\begin{document}

\title{On a Question of J\"{a}gers}

\author{Vassilios Gregoriades}

\date{\today}

\keywords{positive formula, fixed point, hyperarithmetical point, Suslin-Kleene theorem, uniformization}

\subjclass[2010]{03D60, 03E15}

\thanks{Special thanks are owed to {\bf Yiannis Moschovakis} and {\bf Gerhard J\"{a}ger} for valuable discussions.}

\maketitle

\newcounter{mycount}
\newcounter{mycountb}

\begin{abstract}
We show that there exists a positive arithmetical formula $\psi(x,R)$, where $x \in \om$, $R \subseteq \om$, with no hyperarithmetical fixed point. This answers a question of Gerhard J\"{a}ger. As corollaries we obtain results on the proof-theoretic strength of the Kripke-Platek set theory; the fixed points of monotone functions in chain-complete partial orders; the non-Borel uniformization of Borel sets; and the hyperdegrees of fixed points of positive formulae. Further we prove a Suslin-Kleene type result for the specific encoding of the hyperarithmetical sets that we are using.
\end{abstract}

\section{Introduction}

We begin with some comments on \myemph{notation}. We identify the natural numbers with the first infinite ordinal $\om$ and we denote by $\rfn{e}^n$ the $e$-th largest partial recursive function on $\om^n$ to $\om$, where $n > 0$. In fact we will omit the preceding $n$ and write simply $\rfn{e}$ since the domain of the latter will be clear from the context - further  the text will make it clear when the symbol $\{e\}$ refers to the set which contains only $e$.  When we use the term ``$\rfn{e}(k)$" without explicit comment on the domain of $\rfn{e}$ we always mean that $\rfn{e}$ is defined on $k \in \om$\;; for example $\rfn{e}(k) \in A \subseteq \om$ means that $\rfn{e}$ is defined on $k$ and its value at $k$ is a member of $A$.

The set of all finite sequences of natural numbers is denoted by $\om^{< \om}$. We include in $\om^{<\om}$ the empty sequence as well. A typical element of $\om^{< \om}$ will be denoted by $(u_0,\dots,u_{n-1})$, by $n=0$ we mean the empty sequence. Given $u \in \om^{<\om}$ the unique $n$ for which $u = (u_0,\dots,u_{n-1})$ is the \myemph{length} of $u$ and is denoted by $\lh(u)$.

We fix the injective function $\pair{\cdot}: \om^{<\om} \to \om: (u_0,\dots,u_{n-1}) \mapsto p_0^{u_0+1}\cdot \dots \cdot p_{n-1}^{u_{n-1}+1}$, where $p_0 < p_1 < \dots$ is the increasing enumeration of all prime numbers; the empty sequence is mapped by $\pair{\cdot}$ to $1$. If $s = \pair{u_0,\dots,u_{n-1}}$ we say that $s$ \myemph{encodes} $u = (u_0,\dots,u_{n-1})$. We will often view $\om^{< \om}$ as a domain for recursive functions; the latter makes sense via $\pair{\cdot}$. For example a subset $P$ of $\om^{<\om}$ is semirecursive if and only if the set of all $s \in \om$ for which $s = \pair{u_0,\dots,u_{n-1}}$ and $(u_0,\dots,u_{n-1}) \in P$ is semirecursive.

If $X$ and $Y$ are non-empty sets, $P \subseteq X \times Y$, and $x \in X$, we denote by $P(x)$ the $x$-section $\set{y \in P}{(x,y) \in P}$.

The \myemph{Baire space} $\baire$ is the space $^\om\om$ of all sequences of natural numbers together with the product topology and the \myemph{Cantor space} is $\cantor$, where the sequences take values in $2 = \{0,1\}$.

We recall the class \HYP of all \myemph{hyperarithmetical} subsets of $\om$. As it is well-known a set $A \subseteq \om$ is hyperarithmetical exactly when it appears in the constructive hierarchy at a stage below the first non-recursive ordinal $\bold{\ck}$. It is also a well-known fact in effective descriptive set theory that the class \HYP coincides with the class \del \ of all \myemph{effective bi-analytic} sets. The latter is immediate from the Souslin-Kleene Theorem see \cite[7B.4]{yiannis_dst}.

Given $\alpha, \beta \in \baire$ we say that $\bold{\alpha}$ \myemph{is hyperarithmetical in} $\bold{\beta}$ if the set
\[
\set{~\pair{k_0,\dots,k_{n-1}} \in \om}{(\forall i < n)[\alpha(i)=k_i]~}
\]
is a $\HYP(\beta)$ subset of $\om$. The latter is equivalent to saying the singleton $\{\alpha\}$ is a $\del(\beta)$ set. We define 
\begin{align*}
\alpha \hleq \beta \iff& \ \text{$\alpha$ is hyperarithmetical in $\beta$}\\
\alpha \heq \beta \iff& \ \alpha \hleq \beta \ \& \ \beta \hleq \alpha\\
\alpha \hless \beta \iff& \ \alpha \hleq \beta \ \& \ \beta \not \hleq \alpha. 
\end{align*}
The \myemph{hyperdegree of} $\bold{\alpha}$ is the set $\set{\beta}{\alpha \heq \beta}$. We fix in the sequel a subset of the naturals that is $\pii$ and not $\del$, for example \myemph{Kleene's} $\bold{O}$. When we write $O \hleq \alpha$ we mean that the characteristic function $\chi_{O} \in \baire$ of $O$ is $\hleq$-below $\alpha$.

We consider the language $\ca{L}$ of first-order arithmetic and a new unary relation symbol $\tilde{R}$. In the sequel we denote by $\ca{L}(\tilde{R})$ the language obtained by $\ca{L}$ and $\tilde{R}$. We recall that a formula $\psi(x_1,\dots,x_n,\tilde{R})$ in $\ca{L}(\tilde{R})$ is \myemph{positive} in $\tilde{R}$ or simply positive if $\tilde{R}$ does not appear in $\psi$; or it has one of the following forms: $0 \in \tilde{R}$, $ 1 \in \tilde{R}$, $x_i \in \tilde{R}$, $(x_i+1) \in \tilde{R}$, $x_i+x_j \in \tilde{R}$, $x_i  \cdot x_j \in \tilde{R}$, $\varphi \vee \chi$, $\varphi~\&~\chi$,  $\exists x_{n+1}\varphi(x_1,\dots,x_n,x_{n+1},\tilde{R})$, $\forall x_{n+1}\varphi(x_1,\dots,x_n,x_{n+1},\tilde{R})$, where $\varphi$ and $\chi$ are positive.

Evidently a formula $\psi(x,\tilde{R})$ in $\ca{L}(\tilde{R})$ induces the operation 
\[
\Phi_\psi: \ca{P}(\om) \to \ca{P}(\om): A \mapsto \set{y}{\psi(y,A) \ \text{holds}}
\]
(where $\ca{P}(\om)$ is the powerset of $\om$) and if $\psi$ is positive it is easy to see that $\Phi_\psi$ is \textbf{monotone}, \ie if $A \subseteq B \subseteq \om$ then $\Phi_\psi(A) \subseteq \Phi_\psi(B)$. A \textbf{fixed point} of $\psi$ is a set $Q \subseteq \om$ such that for all $y \in \om$ we have
\[
y \in Q \iff \psi(y,Q) \ \text{holds},
\]
equivalently $Q$ is a fixed point of the associated operation $\Phi_\psi$. As it is well-known a monotone operation $\Phi$ has a fixed point (see for example \cite[7C]{yiannis_dst}).

The topic of fixed points in the Kripke-Platek set theory  or in weak fragments of second-order arithmetic has received attention from Gerhard J\"{a}ger and Silvia Steila, see \cite{ jaeger_steila_2018} and \cite{jaeger_2021}. The following question of Gerhard J\"{a}ger was communicated to us through Yiannis Moschovakis: does there exist a positive formula in $\psi(x,\tilde{R})$ in $\ca{L}(\tilde{R})$, which does not have a fixed point that lies in the constructive hierarchy up to the $\ck$-level? In Theorem \ref{theorem positive formula with no HYP fixed point} we show that the answer to this question is affirmative.

\begin{theorem}
\label{theorem positive formula with no HYP fixed point}
Consider the formula of $\ca{L}(\tilde{R})$ defined by
\begin{align*}
\psi(y,\tilde{R}) 
\equiv& \ (\exists a,x,e)\big\{y = \pair{a,x} \ \& \ \big( \ [a = \pair{0,e} \ \& \ x = e ]\\
& \hspace*{20mm} \vee [a = \pair{1,e} \ \& \ (\exists t)(\forall s)[\pair{\rfn{e}(\pair{t,s}),x} \in \tilde{R}] \ \big)\big\}.
\end{align*}
Then $\psi$ is positive and has no hyperarithmetical fixed points, \ie no fixed point of $\psi$ belongs to $L_{\ck}$.
\end{theorem}

\begin{remark}\normalfont
\label{remark following main theorem}
We note that our original version of the preceding result is that the positive formula
\[
\psi'(e,R) \iff \ [e = 1] \ \vee \ [e \neq 1 \ \& \ (\exists n)(\forall m)[\rfn{e}(\pair{n,m}) \in R]]
\]
has no hyperarithmetical fixed points. 

The employment of the formula $\psi$ in favor of $\psi'$ and the subsequent modification of the original proof was suggested to us by Yiannis Moschovakis. While $\psi'$ is slightly simpler to define, we think that $\psi$ provides a shorter and somewhat more elegant proof.

We remark moreover that $\psi$ is $\pii$ on $\pii$ and $\sig$ on $\sig$ as well; the latter means that if $\Gamma$ is on of $\sig$ or $\pii$ then for all sets $P \subseteq \baire \times \baire$ in $\Gamma$ the set $\set{(y,\alpha) \in \om \times \baire}{\psi(y,P(\alpha)) \ \text{holds}}$ is in $\Gamma$ as well. It follows from the Norm Induction Theorem (see \cite[7C.8]{yiannis_dst}) that its least fixed point is a $\pii$ set and its greatest a $\sig$ one. Therefore from Theorem \ref{theorem positive formula with no HYP fixed point} these fixed points are proper $\pii$ and $\sig$ sets.
\end{remark}

In the sequel we investigate some of the consequences of the preceding result.
\smallskip

\myemph{Fixed points in chain-complete partial orders.} Obviously we can identify a subset of $\om$ with a member of $\cantor$ and vice versa. We transfer the subset relation to members of $\cantor$ the usual way,
\[
\alpha \subseteq \beta \iff (\forall n)[\alpha(n) = 1 \ \longrightarrow \ \beta(n) = 1].
\]
A function $f: \cantor \to \cantor$ is \myemph{monotone} if for all $\alpha \subseteq \beta$ we have $f(\alpha) \subseteq f(\beta)$. Obviously the partially ordered space $(\cantor,\subseteq)$ is \myemph{chain-complete}, \ie every chain has a supremum, and as it is well-known every monotone function in a complete partially ordered space has a fixed point. It is natural to ask if there are arithmetical monotone functions $f: \cantor \to \cantor$ without hyperaritmetical fixed points. The answer is provided by Theorem \ref{theorem positive formula with no HYP fixed point}.

\begin{corollary}
\label{corollary monotone arithmetical function with no HYP fixed point}
There exists a $\Sigma^0_4$-recursive monotone function $f: (\cantor,\subseteq) \to (\cantor,\subseteq)$ with no hyperarithmetical fixed points.
\end{corollary}

\begin{proof}
Consider the formula $\psi(x,\tilde{R})$ in Theorem \ref{theorem positive formula with no HYP fixed point}. Clearly the associated operation $\Phi_\psi$ can be identified with the function 
\[
f_\psi: \cantor \to \cantor: f_\psi(\alpha)(m) = 1 \iff \psi(m,A) \ \text{holds},
\]
where $A = \set{n}{\alpha(n)=1}$. Since $\psi$ is positive, the function $f_\psi$ is monotone. Moreover for every fixed point $\alpha$ of $f_\psi$ the set $A = \set{n}{\alpha(n)=1}$ is a fixed point of $\psi$. Hence $f_\psi$ has no hyperarithmetical fixed points. Finally we remark that $\psi$ is a $\Sigma_3$ formula, which implies that $f$ is $\Sigma^0_4$-recursive. 
\end{proof}\smallskip

\myemph{Unprovability in \KP.} As it was communicated to us by G. J\"{a}ger, it was known to him that the Kripke-Platek set theory (\KP) with infinity cannot prove that an arithmetical positive formula {\em with parameters} has fixed points. It is immediate from Theorem \ref{theorem positive formula with no HYP fixed point} that the parameter-free version of the preceding result is also true:

\begin{corollary}
\label{corollary Kripke-Patek for formulae}
There is a positive \tu{(}parameter-free\tu{)} formula $\psi(x,\tilde{R})$ of $\ca{L}(\tilde{R})$ such that
\[
\textup{\KP} + (\textup{Axiom of Infinity}) \ \not \vdash \ \textup{($\psi$ has a fixed point)}.
\]
\end{corollary}

\myemph{Relativization and uniformity}. As usual one can \emph{relativize} the preceding result to a parameter in $\cantor$ in a uniform way: we extend the language $\ca{L}(\tilde{R})$ to $\ca{L}(\tilde{R},\tilde{Q})$ by adding a new unary symbol $\tilde{Q}$ and we define $\tilde{\psi}(x,\tilde{R},\tilde{Q})$ by replacing the $e$-th partial recursive function $\rfn{e}$ in the definition of $\psi$ in Theorem \ref{theorem positive formula with no HYP fixed point} with  $\rfn{e}^{\chi_{\tilde{Q}}}$, where $\chi_{\tilde{Q}}$ stands for the characteristic function of $\tilde{Q}$. In other words the condition ``$\pair{\rfn{e}(\pair{t,s}),x} \in \tilde{R}$" becomes ``there exists an initial segment $u$ of the characteristic function of $Q$ such that $\pair{\rfn{e}(u,\pair{t,s}),x} \in \tilde{R}$. (Notice that here \rfn{e} refers to $\rfn{e}^2$, the $e$-th largest recursive partial function on $\om^2$.) The preceding condition is expressible in $\mathcal{L}(\tilde{R},\tilde{Q})$,
\begin{align*}
&(\exists p = \pair{u_0,\dots,u_{n-1}}) \big [ \ (\forall i < n)\big ( \ u_i = 0,1 \ \& \ [u_i = 1 \leftrightarrow i \in \tilde{Q}] \ \big )\\
& \hspace*{55mm} \ \& \ \pair{\rfn{e}(u,\pair{t,s}),x} \in \tilde{R} \big ].
\end{align*}
Evidently $\tilde{\psi}$ is positive in $\tilde{R}$ (but not in $\tilde{Q}$).

It will become apparent from the proof of Theorem \ref{theorem positive formula with no HYP fixed point} that for every set $C \subseteq \om$, the positive formula $\tilde{\psi}(x,\tilde{R},C)$ of $\ca{L}(\tilde{R})$ with constant $C$ has no $\HYP(C)$-fixed points. 

The relativized version of Corollary \ref{corollary monotone arithmetical function with no HYP fixed point} gives a $\Sigma^0_4$-recursive function $f: \cantor \times \cantor \to \cantor$ such that for all $\gamma \in \cantor$ the section $f_\gamma: (\cantor,\subseteq) \to (\cantor,\subseteq)$ is monotone and has no $\HYP(\gamma)$ fixed points. This in turn has an interesting application in classical descriptive set theory. 

Recall that a set $P^\ast$ \myemph{uniformizes} the set $P \subseteq \ca{X} \times \ca{Y}$, where $\ca{X}$, $\ca{Y}$ are Polish spaces, if $P^\ast \subseteq P$ and for all $x$ for which the section $P_x$ is non-empty there exists exactly one $y \in \ca{Y}$ with $(x,y) \in P^\ast$. It is a prominent question in descriptive set theory to ask whether a given set $P$ that belongs to some pointclass $\Gamma$ can be uniformized by a set $P^\ast$ that is in $\Gamma$ as well.
 
\begin{corollary}
\label{corollary non uniformization with artithmetical f}
There exists a function $f: \cantor \times \cantor \to \cantor$ with the following properties:
\begin{list}{\tu{(}\roman{mycount}\tu{)}}{\usecounter{mycount}}
\item each section $f_\gamma: (\cantor,\subseteq) \to (\cantor,\subseteq)$ is monotone and therefore it has a fixed point;
\item the function $f$ is $\bolds^0_4$-measurable;
\item there is no Borel-measurable function $u: \cantor \to \cantor$ such that $u(\gamma)$ is a fixed point of $f_\gamma$ for all $\gamma \in \cantor$.\smallskip

In particular the set
\[
P = \set{(\gamma,\alpha) \in \cantor \times \cantor}{f(\gamma,\alpha) = \alpha}
\]
is $\boldp^0_4$, has non-empty sections $P(\gamma)$ for all $\gamma \in \cantor$, and cannot be uniformized by any Borel set. 
\end{list}
\end{corollary}

\begin{proof}
Consider the function $f: \cantor \times \cantor \to \cantor$ as in the relativized version of Corollary \ref{corollary monotone arithmetical function with no HYP fixed point}. Since $f$ is $\Sigma^0_4$-recursive it is also $\bolds^0_4$-measurable. Assume toward a contradiction that there is a Borel-measurable function $u: \cantor \to \cantor$ that chooses a fixed point for each section $f_\gamma$. Then $u$ would be $\del(\ep)$-recursive for some $\ep \in \cantor$ and consequently $u(\ep)$ would be a $\HYP(\ep)$ fixed point of $f_\ep$, contradicting the key property of $f$.
\end{proof}

For compactness reasons no function that satisfies Corollary \ref{corollary monotone arithmetical function with no HYP fixed point} can be recursive. We believe that if we replace the Cantor space with the Baire space and with a necessary modification of $\subseteq$ we can indeed obtain a recursive function $f$.

\begin{conjecture}
\label{conjecture recursive f}
There is a $\Pi^0_1$ partial ordering $\preceq$ on $\baire$ such that every $\preceq$-chain has a least upper bound and a function $f: \baire \times \baire \to \baire$ with the following properties:
\begin{list}{\tu{(}\roman{mycount}\tu{)}}{\usecounter{mycount}}
\item the function $f$ is recursive;
\item each section $f_\gamma: (\baire,\preceq) \to (\baire,\preceq)$ is monotone and therefore it has a fixed point;
\item no fixed point of $f_\gamma$ is $\HYP(\gamma)$.
\end{list}
\end{conjecture}

If $f$ satisfies the conclusion of the preceding conjecture it follows as above that the set
\[
P = \set{(\gamma,\alpha) \in \cantor \times \cantor}{f(\gamma,\alpha) = \alpha}
\]
is closed, has non-empty sections $P(\gamma)$ for all $\gamma \in \cantor$, and cannot be uniformized by any Borel set.\smallskip

\myemph{Hypedegrees of fixed points.} With the help of Theorem \ref{theorem positive formula with no HYP fixed point} we can derive some results about the hyperdegrees of the fixed points of positive formulae. 

\begin{corollary}
\label{corollary hyperdegrees of fixed points}
Consider the positive formula $\psi$ of Theorem \ref{theorem positive formula with no HYP fixed point} and let
\[
\fix(\psi) = \set{Q \subseteq \om}{Q \ \text{is a fixed point of} \ \psi}
\]
be the set of all fixed points of $\psi$. Then we have the following.
\begin{enumerate}
\item Every hyperdegree from Kleene's $\mathit{O}$ and above is obtained by some fixed point $Q$, \ie
\[
(\forall \alpha)[ O \hleq \alpha \ \longrightarrow \ (\exists Q \in \fix(\psi))[Q \heq \alpha]].
\]
\item There exists a decreasing sequence $A_0 \supseteq A_1 \supseteq \dots \supseteq A_i \supseteq A_{i+1} \supseteq \dots$ of $\del$ subsets of $\fix(\psi)$ \tu{(}we view the latter as a subset of the Baire space\tu{)} and a sequence of fixed points $(Q_i)_{i \in \om}$ with $Q_i \in A_i$ such that 
\[
(\forall Q \in A_{i+1})[Q \not \hleq Q_i].
\]
In particular the hyperdegree of $Q_i$ does not appear in $A_{i+1}$. Additionally the $Q_i$'s can be chosen so that the relativized Church-Kleene ordinal $\ckr{Q_i}$ equals to $\ck$.
\end{enumerate}
\end{corollary}

\begin{proof}
We remark that the set $\fix(\psi)$ is an arithmetical subset of $\baire$; this is because
\begin{align*}
Q \ \text{is a fixed point of} \ \psi 
\iff& \ (\forall y \in \om)[y \in Q \ \longleftrightarrow \ (\om, +, \cdot, 0,1, R) \vDash \psi(y,Q)],
\end{align*}
and the satisfiability relation $\vDash$ is arithmetical.

As it is well-known every $\del$ set is the recursive injective image of a $\Pi^0_1$ subset of $\baire$, see \cite[4A.7]{yiannis_dst}; hence there exists recursive tree $T$ and a recursive function $\pi: \baire \to \baire$ such that $\fix(\psi)$ is the image of the body $[T]$ of $T$ under $\pi$. Moreover $\pi$ is injective on $[T]$.

Since $\pi$ is recursive we have $\pi(\alpha) \in \del(\alpha)$, and using the injectiveness of $\pi$, it is not hard to see that we also have $\alpha \in \del(\pi(\alpha))$ and therefore $\pi(\alpha) \heq \alpha$ for all $\alpha \in [T]$. Recall also that the $\del$-injective image of a $\del$ set is also $\del$; hence if $B \subseteq [T]$ is $\del$ then $Q: = \pi[B]$ is also $\del$, see \cite[4D.7]{yiannis_dst}. The problem is therefore reduced to a question about $[T]$ rather than $\fix(\psi)$.

From Theorem \ref{theorem positive formula with no HYP fixed point} it follows that $T$ had no $\del$-branches; hence it is a \myemph{Kleene tree}, \ie a recursive tree with $[T] \neq \emptyset$ but $[T] \cap \del = \emptyset$. As it was proved by H. Friedman  \cite{friedman_harvey_borel_sets_and_hyperdegrees} every hyperdegreee from Kleene's $\mathit{O}$ and above occurs in the body of a Kleene tree; hence we proved (1).

Assertion (2) is immediate from the following result (see the proof of \cite[3.13]{gregoriades_classes_of_Polish_spaces_under_effective_Borel_isomorphism}): For every Kleene tree $S$ there exists a Kleene tree $S' \subseteq S$ and some $\gamma \in [S]$ with $\ckr{\gamma} = \ck$ such that $[S'] \cap \del(\gamma) = \emptyset$. Thus we can construct inductively a decreasing sequence $T_0 = T \supseteq T_1 \supseteq \dots \supseteq T_i \supseteq T_{i+1} \supseteq \dots$ of Kleene trees and a sequence $(\gamma_i)_{i \in \om}$ such that $\gamma_i \in [T_i]$ and $[T_{i+1}] \cap \del(\gamma_i) = \emptyset$ for all $i$. We take then $A_i = \pi[[T_i]]$ and $Q_i = \pi(\gamma_i)$, $i \in \om$.
\end{proof}

\begin{question}\normalfont
\label{question analogue to FFT}
We have seen above that $\fix(\psi)$, where $\psi$ is as in Theorem \ref{theorem positive formula with no HYP fixed point}, is no different from the body of a Kleene tree as far as $\del$-injections go. It would be interesting to see if other results on Kleene trees can be transferred to the sets of fixed points of positive formulae. For example it was proved by Fokina - S. Friedman - T\"{o}rnquist  \cite{fokina_friedman_sy_toernquist_the_effective_theory_of_Borel_equivalence_relations} that there are Kleene trees $T,S$ and $(\alpha,\beta) \in [T] \times [S]$ such that $[T] \cap \del(\beta) = \emptyset$ and $[S] \cap \del(\alpha) = \emptyset$.\smallskip

\emph{Do there exist positive formulae $\psi(x,S)$, $\chi(x,S)$ and fixed points $Q_\psi$, $Q_\chi$ of $\psi$ and $\chi$ respectively such that $\fix(\psi) \cap \del(Q_\chi) = \emptyset$ and $\fix(\chi) \cap \del(Q_\psi) = \emptyset$?}
\end{question}

\begin{question}\normalfont
\label{question analogue to minimal hyperdegree}
In \cite{gregoriades_classes_of_Polish_spaces_under_effective_Borel_isomorphism} it is asked if a \emph{minimal} hyperdegree can occur in a Kleene tree, and in the same article it is announced that the latter was solved by G.-Kihara. However G.-Kihara have since found a gap in their argumentation and we take the opportunity to \emph{retract the preceding announcement}. So the above question about hyperdegrees remains open.\smallskip

As above one can ask the similar question for the sets of fixed points of positive formulae.\smallskip

\emph{Does there exist a positive formula $\psi$ such that $\fix(\psi)$ contains a minimal hyperdegree but no $\del$ members?}
\end{question}\smallskip

\section{No fixed point of $\psi$ is hyperarithmetical}

The idea is to show that every fixed point of $\psi$ contains recursively the information of all hyperarithmetical sets, and thus it cannot be hyperarithmetical itself.

\subsection*{Encoding the HYP sets of naturals} For the needs of our proof we use the following natural encoding of the hyperarithmetical sets.

\begin{definition}\normalfont
\label{definition of encoding of HYP}
We consider the following formulae (in the corresponding extensions of $\ca{L}$),
\begin{align*}
\psi_0(a,J) 
\iff& \ (\exists e)[a = \pair{0,e}] \ \vee \ (\exists e)[a = \pair{1,e} \ \& \ (\forall k)[\rfn{e}(k) \in J]]\\
\psi_1(a,x,A) 
\iff& \ (\exists e)[a = \pair{0,e} \ \& \ x = e] \ \vee \ (\exists e)[a = \pair{1,e}\\
& \hspace*{35mm} \& \ (\exists t)(\forall s)[(\rfn{e}(\pair{t,s}),x) \in A]],  
\end{align*}
where above $a,x \in \om$, $J \subseteq \om$ and $A \subseteq \om^2$.

It is evident that the formulae $\psi_0$ and $\psi_1$ are monotone, hence they have a fixed point. We define 
\begin{align*}
\hcode =& \ \text{the least fixed point of} \ \psi_0\\
\hsett =& \ \text{the least fixed point of} \ \psi_1.
\end{align*}
In other words the sets $\hcode$ and $\hsett$ are the \myemph{least sets} satisfying the equivalences
\begin{align}
\label{equation definition of encoding of HYP A}
a \in \hcode 
\iff&  \ (\exists e)[a = \pair{0,e}] \ \vee \ (\exists e)[a = \pair{1,e} \ \& \ (\forall k)[\rfn{e}(k) \in \hcode]]\\
\label{equation definition of encoding of HYP B}
(a,x) \in \hsett
\iff&  \ (\exists e)[a = \pair{0,e} \ \& \ x = e]\\
\nonumber
& \hspace*{8mm} \vee \ (\exists e)[a = \pair{1,e} \ \& \ (\exists t)(\forall s)[(\rfn{e}(\pair{t,s}),x) \in \hsett]]
\end{align}
where $a,x \in \om$.\smallskip

If $A = \hset{a}$ ($=$ the $a$-section of $\hsett$)  we say that {\bf $a$ is an \bold{$\hcode$}-code for \bold{$A$}}.
\end{definition}

\begin{remark}\normalfont
\label{remark following definition of encoding HYP}
\newcounter{myindex}
We make some simple remarks about the preceding definition.
\begin{list}{(\roman{myindex})}{\usecounter{myindex}\leftmargin=2mm}
\item It is evident that the formulae $\psi_0$ and $\psi_1$ are $\pii$ on $\pii$, hence from the Norm Induction Theorem the sets $\hcode$ and $\hsett$ are $\pii$.\smallskip

\item Clearly the $a$-sections \hset{a} of $\hsett$, $a \in \hcode$, satisfy
\begin{align}
\label{equation definition of encoding of HYP C}
\hset{a} =
\begin{cases}
\text{the singleton} \ \{e\}, \ & \ \text{if} \ a = \pair{0,e},\\[1ex]
\textstyle \bigcup_{t} \bigcap_s  \hset{\rfn{e}(\pair{t,s})},  \ & \ \text{if} \ a = \pair{1,e}.
\end{cases}
\end{align}

\item It is useful in the sequel to describe the definition of $\hcode$ ``from below". Define by recursion the family $(\hcode_\xi)_{\xi: \ \text{ordinal}}$ of subsets of $\om$ as follows
\begin{align}\label{equation definition of Ixi}
\begin{cases}
\quad \hcode_0 
=& \hspace*{-2mm} \set{\pair{0,e}}{e \in \om}\\
\quad \hcode_\xi
=& \hspace*{-2mm} \set{\pair{1,e}}{e \in \om \ \& \ (\forall k)(\exists \eta < \xi)[\rfn{e}(k) \in \hcode_\eta]}. 
\end{cases}
\end{align}
Of course the iteration stabilizes at some countable ordinal; in fact at $\ck$.

\myemph{We claim} that $\hcode = \cup_\xi \hcode_\xi$. It is easy to check by induction that $\hcode_\xi \subseteq \hcode$ for all ordinals $\xi$. Conversely, since $\hcode$ is the least fixed point of $\psi_0$, it is enough to show that $\cup_\xi \hcode_\xi$ is a fixed point, \ie that it satisfies (\ref{equation definition of encoding of HYP A}).

The latter is easy to do. If $a \in \cup_\xi\hcode_\xi$ then either $a \in \hcode_0$ in which case $a = \pair{0,e}$ for some $e$; or $a \in \hcode_\xi$ for some $\xi > 0$, in which case for all $k$ there exists $\eta < \xi$ such that $\rfn{e}(k) \in \hcode_\eta \subseteq \cup_\zeta \hcode_\zeta$. Hence the set $\cup_\xi\hcode_\xi$ satisfies the direct implication of (\ref{equation definition of encoding of HYP A}). For the converse implication it is clear that if $a$ satisfies the first conjunct then $a \in \hcode_0$, and if $a$ satisfies the second one then $a = \pair{1,e}$ for some $e$ and for all $k$ there exists some $\eta_k$ such that $\rfn{e}(k) \in \hcode_{\eta_k}$; then $a \in \hcode_\xi$, where $\xi = \sup\set{\eta_k}{k \in \om}+1$. 
\end{list}

We \myemph{fix once and for all} the sequence $(\hcode_\xi)_\xi$ defined in (\ref{equation definition of Ixi}).
\end{remark}

\begin{definition}\normalfont
\label{definition norm on hcode}
Given $a \in \hcode$ we define \myemph{the norm} $\bold{\normI{a}}$ \myemph{of} $\bold{a}$ as follows: 
\[
\normI{a} = \ \text{the least $\xi$ such that} \ a \in \hcode_\xi. 
\]

It is evident that
\begin{align}
\label{equation definition norm on hcode A}
\normI{a} = 0 \iff& (\exists e)[a = \pair{0,e}] \quad \text{and}\\
\label{equation definition norm on hcode B}
\normI{a} > 0 \ \& \ a = \pair{1,e} \Longrightarrow& \ (\forall k)[\rfn{e}(k) \in \hcode \ \& \ \normI{\rfn{e}(k)} < \normI{a}].
\end{align}
\end{definition}\smallskip

The main aim is to show that every hyperarithmetical subset of $\om$ is of the form $\hset{a}$ for some $a \in I$. In fact with some more work one can show that the sets $\hset{a}$, $a \in I$, are exactly the hyperarithmetical ones, but this is not necessary for our purposes. 

\begin{lemma}
\label{lemma provides encoding of hyp}
For every hyperarithmetical set $H \subseteq \om$ there exists some $a \in \hcode$ such that $$H = \hset{a} = \set{x}{(a,x) \in \hsett}.$$
\end{lemma}

Using the preceding result we can prove Theorem \ref{theorem positive formula with no HYP fixed point} as follows.

\begin{lemma}
\label{lemma reduction identity}
If $Q$ is a fixed point of $\psi$ then for all $a \in \hcode$ and all $x \in \om$ we have
\[
x \in \hset{a} \iff \pair{a,x} \in Q.
\]
\end{lemma}

\begin{proof}
This is done by induction on $\normI{a}$. Suppose that $a$ is in $I$ and that $\normI{a} = 0$. Then from (\ref{equation definition norm on hcode A}) and (\ref{equation definition of encoding of HYP B}) we get $a = \pair{0,e}$ for some $e \in \om$ and $\hset{a}$ is the singleton $\{e\}$. We then compute
\begin{align*}
x \in \hset{a} 
\iff& \ x = e\\
\iff& \ \psi(\pair{\pair{0,e},x},Q) \hspace*{5mm} \text{(definition of $\psi$)}\\
\iff& \ \pair{\pair{0,e},x} \in Q \hspace*{7mm} \text{($Q$ is a fixed point)}\\
\iff& \ \pair{a,x} \in Q
\end{align*}
for all $x \in \om$.

Assume now that $\normI{a} > 0$ and that the assertion is true for all $b \in I$ with $\normI{b} < \normI{a}$. Then $a$ has the form $\pair{1,e}$ for some $e$, and from (\ref{equation definition norm on hcode B}) we have $\rfn{e}(k) \in \hcode$ and $\normI{\rfn{e}(k)} < \normI{a}$ for all $k$. We compute
\begin{align*}
x \in \hset{a} 
\iff& \ x \in \hset{\pair{1,e}}\\
\iff& \ (\exists t)(\forall s)[x \in \hset{\rfn{e}(\pair{t,s})}] \hspace*{7mm} \text{(from (\ref{equation definition of encoding of HYP C}))}\\
\iff& \ (\exists t)(\forall s)[\pair{\rfn{e}(\pair{t,s}),x} \in Q] \hspace*{5mm} \text{(by the ind. hyp.)}\\
\iff& \ \psi(\pair{\pair{1,e},x},Q) \hspace*{23mm} \text{(definition of $\psi$)}\\
\iff& \ \pair{\pair{1,e},x} \in Q \hspace*{25mm} \text{($Q$ is a fixed point)}\\
\iff& \ \pair{a,x} \in Q,
\end{align*}
for all $x \in \om$. This concludes the inductive step.
\end{proof}

\begin{lemma}
\label{lemma the fixed point is not hyp}
No fixed point $Q$ of $\psi$ is hyperarithmetical.
\end{lemma}

\begin{proof}
Suppose that $Q$ is a fixed point of $\psi$. If it were $Q \in \HYP$ then the set
\[
P = \set{x \in \om}{(\exists a,k)[x = \pair{a,k} \ \& \ \pair{a,x} \not \in Q]}
\]
would be hyperarithmetical as well. From Lemma \ref{lemma provides encoding of hyp} there would be some $a^\ast \in \hcode$ such that $P = \hset{a^\ast}$. We compute
\begin{align*}
\pair{a^\ast,0} \in \hset{a^\ast} \iff& \ \pair{a^\ast,0} \in P\\ 
\iff& \ \pair{a^\ast,\pair{a^\ast,0}} \not \in Q\\
\iff& \ \pair{a^\ast,0} \not \in \hset{a^\ast},
\end{align*}
a contradiction, where in the last equivalence we used Lemma \ref{lemma reduction identity}. 
\end{proof}

It remains to prove Lemma \ref{lemma provides encoding of hyp}. In the first draft of this article that is available on the arXiv we prove the latter using the effective version of the notion of a $\sigma$-algebra; to be more specific we show that the family $\set{\hset{a}}{a \in \hcode}$ is an {\em effective Borel $\sigma$-algebra} \cite[7B.7]{yiannis_dst} (the latter uses the term ``effective $\sigma$-field"). Intuitively an effective Borel $\sigma$-algebra is a family of sets, which contains recursively-uniformly the basic sets (in the case of $\om$ these are the singletons), and is recursively-uniformly closed under the operations of complement and countable union.

As it was proved by Kleene the family of all $\del = \HYP$ subsets of $\om$ is the least effective Borel $\sigma$-algebra on $\om$, see \cite[7B.7]{yiannis_dst}. This implies that $\HYP \subseteq \set{\hset{a}}{a \in \hcode}$. 

In this updated version of the article we prove Lemma \ref{lemma provides encoding of hyp} as a corollary to a  Suslin-Kleene type result, which we find interesting in its own right.

We fix a set $G \subseteq \om \times \om$, which is \myemph{universal} for the family of all $\sig$ subsets of $\om$, \ie $G$ is $\sig$ and for all $\sig$ sets $A \subseteq \om$ there is $e \in \om$ such that $A = G(e)$. We think of such an $e$ as a \myemph{code} for $A$. 

Given two disjoint sets $A$ and $B$ we say that a set $C$ \myemph{separates $\bold{A}$ from $\bold{B}$} if $A \subseteq C$ and $C \cap B = \emptyset$. We will show that the separation of disjoint $\sig$ sets can be witnessed recursively-uniformly using the $\hcode$-codes.

\begin{proposition}
\label{proposition separation}
If $A, B \subseteq \om$ are disjoint $\sig$ sets then there is $a \in \hcode$ such that the set $C = \hset{a}$ separates $A$ from $B$. 

In fact the preceding separation can be done uniformly in the codes, \ie there is a recursive function $u: \om \times \om \to \om$ such that whenever $G(e_0) \cap G(e_1) = \emptyset$ then $u(e_0,e_1) \in \hcode$ and $\hset{u(e_0,e_1)}$ separates $G(e_0)$ from $G(e_1)$.
\end{proposition}

Lemma \ref{lemma provides encoding of hyp} follows easily from the preceding result. Given a \HYP set $H \subseteq \om$ the sets $A : = H$ and $B := \om \setminus H$ are disjoint $\sig$ sets, and if $\hset{a}$ separates $A$ from $B$ we necessarily have $H = A = \hset{a}$.\smallskip

We continue with the \myemph{proof of Proposition \ref{proposition separation}}. One can do this indirectly using the statement of the Suslin-Kleene Theorem \cite[7B.4]{yiannis_dst}; the latter gives a recursive way to transfer codes for two disjoint $\sig$ sets to a {\em Borel} code for a separating $\del$ set. So we need is a recursive way that transfers a recursive  Borel code for the given $\del$ subset of $\om$ to a code  $a \in \hcode$ for the same set. This is feasible but in effect it is a repetition of our initial proof with the effective Borel $\sigma$-algebras.

Instead we employ the proof of the Suslin-Kleene Theorem to get the $\hcode$-code of the separating set. In fact the $\hcode$-code will arise somewhat more naturally, since the separating set in the Suslin-Kleene Theorem is obtained by applying successively the operation $\textstyle \bigcup_{t \in \om} \bigcap_{s \in \om}$ like we do with $\hsett$.\smallskip

To ensure the uniformity in the separation we need some auxiliary results. First we fix the function $u_1: \om \to \om: u_1(e) = \pair{0,e}$. It is evident from (\ref{equation definition of encoding of HYP A}) and (\ref{equation definition of encoding of HYP B}) that $u_1(e)$ is an $\hcode$-code for the singleton $\{e\}$.

Further we can find some $e^\ast \in \om$ such that
\[
\rfn{e^\ast}(\pair{t,s}) =
\begin{cases}
u_1(0), \ & \ \text{if} \ s=0\\[1ex]
u_1(1), \ & \ \text{if} \ s \neq 0,
\end{cases}
\]
for all $t,s$. Then  $\textstyle \bigcup_{t \in \om} \bigcap_{s \in \om} \hset{\rfn{e^\ast}(\pair{t,s})} = \{0\} \cap \{1\} = \emptyset$. Hence the number 
\[
e_\emptyset : = \pair{1,e^\ast}
\]
is an $\hcode$-code for the empty set.

\begin{claim}\normalfont
\label{claim auxiliary functions for uniformity}
The following hold.
\begin{list}{(\roman{myindex})}{\usecounter{myindex}\leftmargin=2mm}


\item For every natural $k \geq 1$ and every semirecursive set $P \subseteq \om^k \times \om$ there exists a recursive function $u_{P}: \om^k \to \om$ such that  for all $y \in \om^k$ the number $u_{P}(y)$ is an $\hcode$-code for the $y$-section $P(y)$ of $P$.

\item There exist recursive functions $u_\vee, u_\wedge: \om^2 \to \om$ such that for all $e_0, e_1 \in \hcode$ the numbers $u_\vee(e_0,e_1)$ and $u_\wedge(e_0,e_1)$ are $\hcode$-codes for the union $\hset{e_0} \cup \hset{e_1}$ and the intersection $\hset{e_0} \cap \hset{e_1}$ respectively.
\end{list}
\end{claim}

\begin{proof}
For (i) we consider a semirecursive $P \subseteq \om^k \times \om$; then there is a recursive function $f: \om \to \om^k \times \om$ that enumerates it,
\[
(y,x) \in P \iff \exists t \ f(t) = (y,x).
\]

We consider the projection functions $\pr_1: \om^k \times \om \to \om^k: (y,x) \mapsto y$ and $\pr_2: \om^k \times \om \to \om: (y,x) \mapsto x$.

Then we have
\begin{align*}
x \in P(y) 
\iff& \ \exists t  \ f(t) = (y,x)\\
\iff& \ \exists t \ (x \in \{\pr_2(f(t))\} \ \&  \ y = \pr_1(f(t))).
\end{align*}

There is some $e^\ast$ such that
\[
\rfn{e^\ast}(y,\pair{t,s}) =
\begin{cases}
u_1(\pr_2(f(t))), & \ \text{if} \ y = \pr_1(f(t))\\
e_\emptyset, & \ \text{if} \ y \neq \pr_1(f(t)),
\end{cases}
\]
for all $y,t,s$, where $u_1$ and $e_\emptyset$ are as above. We then have that\allowdisplaybreaks
\begin{align*}
P(y) 
=& \ \textstyle \bigcup_{t \in \om, y = \pr_1(f(t))} \ \{\pr_2(f(t))\}\\
=& \ \textstyle \bigcup_{t \in \om, y = \pr_1(f(t))} \ \hset{u_1(\pr_2(f(t)))}\\
=& \ \textstyle \bigcup_{t \in \om} \ \hset{\rfn{e^\ast}(y,\pair{t,0})}\\
& \ \hspace*{25mm} \text{(since $e_\emptyset$ is the code for $\emptyset$)}\\
=& \ \textstyle \bigcup_{t \in \om} \bigcap_{s \in \om} \ \hset{\rfn{e^\ast}(y,\pair{t,s})}\\
& \ \hspace*{15mm} \text{(since the value $\rfn{e^ \ast}(y,\pair{t,s})$ is independent of $s$)}.\\
\end{align*}

From the $s$-$m$-$n$ Theorem there exists a recursive function $S$ such that
\[
\rfn{e^\ast}(y,\pair{t,s}) = \rfn{S(e^\ast,y)}(\pair{t,s})
\]
for all $y,t,s$. Since $\rfn{e^\ast}(y,\pair{t,s}) \in \hcode$ for all $y,t,s$, it follows from (\ref{equation definition of encoding of HYP A}) that $\pair{1,S(e^\ast,y)} \in \hcode$ for all $y$. Moreover\allowdisplaybreaks
\begin{align*}
P(y) 
=& \ \textstyle \bigcup_{t \in \om} \bigcap_{s \in \om} \ \hset{\rfn{e}(y,\pair{t,s})}\\
=& \ \textstyle \bigcup_{t \in \om} \bigcap_{s \in \om} \ \hset{ \rfn{S(e^\ast,y)}(\pair{t,s})}\\
=& \ \textstyle \hset{\pair{1,S(e^\ast,y)}}
\end{align*}
for all $y$, where in the last equality above we used (\ref{equation definition of encoding of HYP C}).

Therefore if we put  $u_P(y) = \pair{1,S(e^\ast,y)}$ we have the required property.\smallskip

For (ii) from there is some $e^{\ast}_1$ such that
\[
\rfn{e^{\ast}_1}(\pair{t,s}) =
\begin{cases}
e_0, & \ \text{if} \ t = 0,\\
e_1, &  \ \text{if} \ t \neq 0,
\end{cases}
\]
for all $s,t$. As before we have
\[
\hset{e_0} \cup \hset{e_1} 
= \textstyle \bigcup_{t} \hset{\rfn{e^{\ast}_1}(\pair{t,0})}
=  \bigcup_{t} \bigcap_s \hset{\rfn{e^{\ast}_1}(\pair{t,s})}
= \hset{\pair{1,e^{\ast}_1}}.
\]
Therefore we take $u_\vee(e_0,e_1) = \pair{1,e^{\ast}_1}$. The assertion about $u_\wedge$ is essentially the same, we just exchange $t$ with $s$ in the distinction of cases.
\end{proof}

We proceed with more on notation. Recall the set $\om^{< \om}$ of all finite sequences of naturals. Given $u = (u_0,\dots,u_{n-1})$, $v = (v_0,\dots,v_{m-1}) \in \om^{< \om}$, we denote by $\cn{u}{v}$ the \myemph{concatenation} of $u$ and $v$, \ie
\[
\cn{u}{v} = (u_0,\dots,u_{n-1},v_0,\dots,v_{m-1}).
\]

By $u \sqsubseteq v$ we mean that $u$ is an \myemph{initial segment} of $v$, \ie $n \leq m$ and for all $i < n$ we have $u_i = v_i$. We extend $\sqsubseteq$ to a relation between finite sequences and elements of the Baire space the natural way,
\[
u \sqsubseteq \gamma \iff \forall i < n \ u_i = \gamma(i), 
\]
where $u = (u_0,\dots,u_{n-1})$ and $\gamma \in \baire$. We do the same for the concatenation,
\[
\cn{u}{\gamma} = (u_0,\dots,u_{n-1},\gamma(0),\gamma(1), \gamma(2),\dots).
\]

Recall that a \myemph{tree} on the naturals is a non-empty set $T$ of finite sequences of naturals that is closed from below under $\sqsubseteq$,
\[
\forall u, v \  [ (u \sqsubseteq v \ \& \ v \in T) \ \longrightarrow \ u \in T ].
\]
The empty sequence is a member of every tree. An \myemph{infinite branch} of $T$ is an element $\alpha \in \baire$ for which $(\alpha(0),\dots, \alpha(n)) \in T$ for all $n$. The set of all infinite branches of $T$ is the \myemph{body} of $T$ and is denoted by $[T]$.

A tree is \myemph{recursive} if its characteristic function $\chi_T: \om^{< \om} \to \{0,1\}$ is recursive.

We also consider \myemph{trees of pairs}, these are sets of finite sequences of pairs of naturals. A typical element of such a tree is of the form
\[
((u_0,v_0),\dots,(u_{m-1},v_{m-1})).
\]
We will omit the inner parenthesis.  We will also denote the latter sequence of pairs by $(u,v)$, where $u = (u_0,\dots,u_{m-1})$ and $v = (v_0,\dots,v_{m-1})$. The infinite branches of trees of pairs are defined analogously. The recursive trees of pairs are defined via some standard recursive bijection between $\om^2$ and $\om$, or simply by using $\pair{\cdot}$.

Next we turn our attention to the separation. Recall the set $G \subseteq \om \times \om$, which is universal for the $\sig$ subsets of $\om$. Since $G$ is $\sig$, as it is well-known there exists a recursive tree $T$ on the naturals such that
\[
n \in G(e) \iff \exists \gamma \ \cn{(e,n)}{\gamma} \in [T]
\]
for all $e$, $n$. We define
\[
G_u(e) = \set{n}{\exists \gamma \ (u \sqsubseteq \gamma \ \& \ \cn{(e,n)}{\gamma} \in [T])},
\]
where $u \in \om^{<\om}$.

It is easy to verify that
\[
G_u(e) = \textstyle \bigcup_{i \in \om} G_{\cn{u}{(i)}}(e).
\]

For all naturals $e_0$, $e_1$ and $n$ we define the tree of pairs $J_n^{e_0,e_1}$,
\[
(u_0,v_0,\dots,u_{m-1},v_{m-1}) \in J_n^{e_0,e_1} \iff \cn{(e_0,n)}{u} \in T \ \& \ \cn{(e_1,n)}{v} \in T,
\]
where $u = (u_0,\dots,u_{m-1})$ and $v = (v_0,\dots,v_{m-1})$. It is evident that the $J_n(e_0,e_1)$'s are recursive and in fact they are recursive uniformly on $e_0,e_1$ and $n$.

\begin{claim}\normalfont
\label{claim local separation}
For all $e_0$, $e_1$ with $G(e_0) \cap G(e_1) = \emptyset$, and all $n$ there exists a family $(C^{e_0,e_1}_{n,u,v})_{(u,v) \in J_n^{e_0,e_1}}$  of $\HYP$ subsets of $\om$ such that for all $(u,v) \in J_n^{e_0,e_1}$ the set $C^{e_0,e_1}_{n,u,v}$ separates $\{n\} \cap G_u(e_0)$ from $\{n\} \cap G_v(e_1)$.

Further there exists a recursive function 
\[
g: \om^3 \times \om^{< \om} \times \om^{< \om} \longrightarrow \om
\]
such that if $G(e_0) \cap G(e_1) = \emptyset$ then $g(e_0,e_1,n,u,v)$ is an $\hcode$-code for $C^{e_0,e_1}_{n,u,v}$ for all $(u,v) \in J_n^{e_0,e_1}$ and all $e_0$,$e_1$,$n$,$u$,$v$.
\end{claim}

\begin{proof}
First we explain how to define the separating sets and then how to get the coding function. For the moment we fix some naturals $e_0$, $e_1$ with $G(e_0) \cap G(e_1) = \emptyset$ and we suppress the symbols $e_0$, $e_1$ in the definition of the separating sets and the trees $J_n^{e_0,e_1}$.

For all $n$ the tree $J_n$ is \emph{well-founded}, \ie it has no infinite branch. Else there would be $\alpha$ and $\beta$ in the Baire space such that $\cn{(e_0,n)}{\alpha} \in [T]$, $\cn{(e_1,n)}{\beta} \in [T]$; so would have $n \in G(e_0) \cap G(e_1)$, a contradiction.

Further we fix some $n$. The definition of the family $(C_{n,u,v})_{(u,v) \in J_n}$ is done using backwards induction on the well-founded tree $J_n$. In the process we also define auxiliary sets $D^{i,j}_{n,u,v}$ for $i,j \in \om$.\smallskip

\emph{Induction.} Assume that $C_{n,u',v'}$ has been defined and separates $\{n\} \cap G_{u'}(e_0)$ from $\{n\} \cap G_{v'}(e_1)$ for all $(u', v') \in J_n$ extending properly $(u,v) \in J_n$. We do the same for $(u,v)$. Given $i,j \in \om$ we have three cases.\smallskip

\emph{Case Ind1.} $\cn{(e_0,n)}{\cn{u}{(i)}} \in T$ and $\cn{(e_1,n)}{\cn{v}{(j)}} \not \in T$. Then we define $D^{i,j}_{n,u,v} = \{n\}$.\smallskip 

\emph{Case Ind2.} $\cn{(e_0,n)}{\cn{u}{(i)}} \in T$ and $\cn{(e_1,n)}{\cn{v}{(j)}} \in T$. We take $u' = \cn{u}{(i)}$ and $v '= \cn{v}{(j)}$. Then the pair $(u',v')$ extends $(u,v)$ properly and belongs to $J_n$. So $C_{n,u',v'}$ has been defined and we take $D^{i,j}_{n,u,v} = C_{n,\cn{u}{(i)},\cn{v}{(j)}}$.\smallskip

\emph{Case Ind3.} $\cn{(e_0,n)}{\cn{u}{(i)}} \not \in T$. Then we take $D^{i,j}_{n,u,v} = \emptyset$.\smallskip

Having defined the $D^{i,j}_{n,u,v}$'s we take then
$
C_{n,u,v} = \textstyle \bigcup_{i \in \om} \bigcap_{j \in \om} D^{i,j}_{n,u,v}.\smallskip
$

We show that $C_{n,u,v}$ separates $\{n\} \cap G_u(e_0)$ from $\{n\} \cap G_v(e_1)$. The latter means that if $n \in G_u(e_0)$ then $n \in C_{n,u,v}$, and if $n \in G_v(e_1)$ then $n \not \in C_{n,u,v}$. For the former, assume that $n \in G_u(e_0)$, then there is $\gamma \in \baire$ such that $u \sqsubseteq \gamma$ and $\cn{(e_0,n)}{\gamma} \in [T]$. We take $i = \gamma(\lh(u))$, \ie $i$ is the first element of $\gamma$ after $u$. In particular $\cn{(e_0,n)}{\cn{(u_0,\dots,u_{\lh(u)-1})}{(i)}} \in T$; then for all $j$ we are either in Case Ind1 or in Case Ind2. In the former it is obvious that $n \in D^{i,j}_{n,u,v} = \{n\}$; in the latter we use the induction hypothesis that $C_{n,\cn{u}{(i)},\cn{v}{(j)}}$ separates $\{n\} \cap G_{\cn{u}{(i)}}(e_0)$ from $G_{\cn{v}{(j)}}(e_1)$. Since $n \in G_{\cn{u}{(i)}}(e_0)$ we have in particular that $n \in C_{n,\cn{u}{(i)},\cn{v}{(j)}} = D^{i,j}_{n,u,v}$. For the second assertion of the separation, we assume that $n \in G_v(e_1)$ and we take $\delta \in \baire$ such that $v \sqsubseteq \delta$ and $\cn{(e_1,n)}{\delta} \in [T]$. We consider some $i \in \om$; then for $j = \delta(\lh(v))$ we have $\cn{(e_1,n)}{\cn{v}{(j)}} = \cn{(e_1,n)}{(v_0,\dots,v_{\lh(v)-1}, \delta(\lh(v)))} \in T$, in particular we are not in Case Ind1. If we are in Case Ind3 then $ n \not \in D^{i,j}_{n,u,v} = \emptyset$. In Case Ind2 the set $C_{n,\cn{u}{(i)},\cn{v}{(j)}} = D^{i,j}_{n,u,v}$ separates $\{n\} \cap G_{\cn{u}{(i)}}(e_0)$ from $\{n\} \cap G_{\cn{v}{(j)}}(e_1)$. Since $n \in G_{\cn{v}{(j)}}(e_1)$ we have in particular that $n \not \in D^{i,j}_{n,u,v}$.\bigskip

The above shows how to define the separating sets $C^{e_0,e_1}_{n,u,v}$, for $(u,v) \in J_n^{e_0,e_1}$, and all $e_0$,$e_1$,$n$ for which $G(e_0) \cap G(e_1) = \emptyset$. Next we show how to obtain the recursive function $g$. We denote by $D^{e_0,e_1,i,j}_{n,u,v}$ the auxiliary sets $D^{i,j}_{n,u,v}$, which are defined at the instance $e_0$, $e_1$. It is evident from the previous definitions that the separating and the auxiliary sets satisfy the following conditions for all $e_0$, $e_1$, $n$ with $G(e_0) \cap G(e_1) = \emptyset$,
\begin{align}
\label{equation properties of separating A}
&\big(\text{$\cn{(e_0,n)}{\cn{u}{(i)}} \in T$ and $\cn{(e_1,n)}{\cn{v}{(j)}} \not \in T$}\big) \ \Longrightarrow \ D^{e_0,e_1,i,j}_{n,u,v} = \{n\}\\
\label{equation properties of separating B}
&\big(\text{$\cn{(e_0,n)}{\cn{u}{(i)}} \in T$ and $\cn{(e_1,n)}{\cn{v}{(j)}} \in T$}\big) \ \Longrightarrow \ D^{e_0,e_1,i,j}_{n,u,v} = C^{e_0,e_1}_{n,\cn{u}{(i)},\cn{v}{(j)}}\\
\label{equation properties of separating C}
&\big(\text{$\cn{(e_0,n)}{\cn{u}{(i)}} \not \in T$}\big) \ \Longrightarrow \ D^{e_0,e_1,i,j}_{n,u,v} = \emptyset\\
\label{equation properties of separating D}
& \hspace*{40mm} C^{e_0,e_1}_{n,u,v} = \textstyle \bigcup_{i \in \om} \bigcap_{j \in \om} D^{e_0,e_1,i,j}_{n,u,v}.
\end{align}
 
Let us denote by $P_1(e_0,e_1,n,u,v,i,j)$, $P_2(e_0,e_1,n,u,v,i,j)$,  and $P_3(e_0,n,u,i)$ the conditions in the premise of (\ref{equation properties of separating A}), (\ref{equation properties of separating B}), and (\ref{equation properties of separating C}) respectively, e.g.~ $P_3(e_0,n,u,i)$ means that $\cn{(e_0,n)}{\cn{u}{(i)}} \not \in T$. Obviously these are recursive and for all $e_0$,$e_1$,$n$,$u$,$v$,$i$,$j$ exactly one of $P_1$, $P_2$, $P_3$ holds, regardless of $G(e_0) \cap G(e_1)$.

Recall the function $u_1: \om \to \om$ such that $u_1(n)$ is an $\hcode$-code for the singleton $\{n\}$, and the $\hcode$-code $e_\emptyset$, which is an $\hcode$-code for the empty set.\smallskip

We define the function $\varphi: \om^4 \times \om^{< \om} \times \om^{<\om} \times \om \longrightarrow \om$ as follows:
\begin{align*}
\varphi(e,e_0,e_1,n,u,v,\pair{i,j}) =
\begin{cases}
u_1(n), & \ P_1(e_0,e_1,n,u,v,i,j),\\
\pair{1,S(e,e_0,e_1,n,\cn{u}{(i)},\cn{v}{(j)})}, &  \ P_2(e_0,e_1,n,u,v,i,j),\\
e_\emptyset, & \ P_3(e_0,n,u,i),
\end{cases}
\end{align*}
where $S: \om^4 \times \om^{<\om} \times \om^{<\om} \to \om$ is the function of the $s$-$m$-$n$ Theorem. (When the last argument of $\varphi$ is not of the form $\pair{i,j}$ we just assign $0$.)

Obviously $\varphi$ is recursive and from the Recursion Theorem there is some $e^\ast$ such that
\[
\varphi(e^\ast,e_0,e_1,n,u,v,\pair{i,j}) = \rfn{e^\ast}(e_0,e_1,n,u,v,\pair{i,j}) = \rfn{S(e^\ast,e_0,e_1,n,u,v)}(\pair{i,j})
\]
for all $e_0$,$e_1$,$n$,$u$,$v$,$i$,$j$.\smallskip

We show that 
\begin{align}
\label{equation phi is an I code}
\text{if $G(e_0) \cap G(e_1) = \emptyset$ then $\varphi(e^\ast,e_0,e_1,n,u,v,\pair{i,j})$ is an $\hcode$-code for $D^{e_0,e_1,i,j}_{n,u,v}$,}
\end{align}
for all $n,u,v,i,j$. If we do this then we will have that $\rfn{S(e^\ast,e_0,e_1,n,u,v)}(\pair{i,j})$ is an $\hcode$-code for $D^{e_0,e_1,i,j}_{n,u,v}$ and hence that $\pair{1,S(e^\ast,e_0,e_1,n,u,v)}$ is an $\hcode$-code for 
\[
\textstyle \bigcup_i \bigcap_j D^{e_0,e_1,i,j}_{n,u,v} = C^{e_0,e_1}_{n,u,v},
\]
where in the last equality we used (\ref{equation properties of separating D}). Hence the required function $g$ is given by
\[
g(e_0,e_1,n,u,v) = \pair{1,S(e^\ast,e_0,e_1,n,u,v)}.
\]

Notice that $g$ is recursive and total; in particular we do not need to assume that the sets $G(e_0)$ and $G(e_1)$ are disjoint.

To prove (\ref{equation phi is an I code}) we consider  $e_0$,$e_1$,$n$ with $G(e_0) \cap G(e_1) = \emptyset$ and we show by backwards induction on the well-founded tree $J^{e_0,e_1}_n$ that $\varphi(e^\ast,e_0,e_1,n,u,v,\pair{i,j})$ is an $\hcode$-code for $D^{e_0,e_1,i,j}_{n,u,v}$.

Suppose that our assertion is true for all $(u',v') \in J^{e_0,e_1}_n$ that extend a given $(u,v) \in J^{e_0,e_1}_n$ properly. We show the same for $(u,v)$. We consider $i,j \in \om$. If one of $P_1(e_0,e_1,n,u,v,i,j)$, $P_3(e_0,n,u,i)$ holds, then $\varphi(e^\ast,e_0,e_1,n,u,v,\pair{i,j})$ is an $\hcode$-code for $D^{e_0,e_1,i,j}_{n,u,v}$ according to  (\ref{equation properties of separating A}) and  (\ref{equation properties of separating C}). So we assume that $P_2(e_0,e_1,n,u,v,i,j)$ holds.

We take $u' = \cn{u}{(i)}$ and $v' = \cn{v}{(j)}$ and we have that $(u',v') \in J^{e_0,e_1}_n$; further the pair $(u',v')$ extends $(u,v)$ properly. From the induction hypothesis we have for all $i', j'$ that
$
\varphi(e^\ast,e_0,e_1,n,u',v',\pair{i',j'})
$
is an $\hcode$-code for $D^{e_0,e_1,i',j'}_{n,u',v'}$. Since for all $i', j'$,
\[
\rfn{S(e^\ast,e_0,e_1,n,u',v')}(\pair{i',j'}) = \varphi(e^\ast,e_0,e_1,n,u',v',\pair{i',j'})
\]
we have that $\pair{1,S(e^\ast,e_0,e_1,n,u',v')}$ is an $\hcode$-code for the set
\[
\textstyle \bigcup_{i'} \bigcap_{j'} D^{e_0,e_1,i',j'}_{n,u',v'} = C^{e_0,e_1}_{n,u',v'} = C^{e_0,e_1}_{n,\cn{u}{(i)},\cn{v}{(j)}} = D^{e_0,e_1,i,j}_{n,u,v},
\]
where in the first and the last of the above equalities we used (\ref{equation properties of separating D}) and (\ref{equation properties of separating B}) respectively and also the fact that $P_2(e_0,e_1,n,u,v,i,j)$ holds.

Moreover from the definition of $\varphi$,
\[
\pair{1,S(e^\ast,e_0,e_1,n,u',v')} = \pair{1,S(e^\ast,e_0,e_1,n,\cn{u}{(i)},\cn{v}{(j)})} = \varphi(e^\ast,e_0,e_1,n,u,v,\pair{i,j}).
\]

Hence $\varphi(e^\ast,e_0,e_1,n,u,v,\pair{i,j})$ is an $\hcode$-code for $D^{e_0,e_1,i,j}_{n,u,v}$ and the inductive step is complete. This concludes the proof of the claim.
\end{proof}

Going back to the proof of Proposition \ref{proposition separation} we fix for the moment $e_0$, $e_1$ with $G(e_0) \cap G(e_1) = \emptyset$ and we consider the sets $C^{e_0,e_1}_{n,u,v}$ for all $n$ and all $(u,v) \in J_n^{e_o,e_1}$, as in Claim \ref{claim local separation}. We define the set $C^{e_0,e_1}_\emptyset\subseteq \om$ as follows,
\begin{align*}
n \in C^{e_0,e_1}_\emptyset \iff \exists i \ \big [ (e_0,n,i) \in T \ \& \ \forall j \ \big ( (e_1,n,j) \in T \ \longrightarrow \ n \in C^{e_0,e_1}_{n,(i),(j)} \big ) \big ].
\end{align*}

Then it is easy to see that $C^{e_0,e_1}_\emptyset$ separates $G(e_0)$ from $G(e_1)$.\smallskip

To finish the proof we need to define the recursive uniformity function $u: \om \times \om \to \om$. First we notice that for the fixed $e_0$, $e_1$ it holds
\begin{align*}
n \in C^{e_0,e_1}_\emptyset 
\iff& \  \exists k \exists i \forall j \ \big[ (e_0,n,i) \in T \ \& \big( (e_1,n,j) \not \in T \ \vee \ n \in C^{e_0,e_1}_{k,(i),(j)} \cap \{k\} \big) \big]
\end{align*}

It follows that there are recursive conditions $R_1(e_0,e_1,i,j,n)$ and $R_2(e_0,i,n)$ such that
\begin{align*}
n \in C^{e_0,e_1}_\emptyset 
\iff& \ \exists k \exists i \forall j \ \big [  (R_1(e_0,e_1,i,j,n) \vee (R_2(e_0,i,n) \ \& \ n \in C^{e_0,e_1}_{k,(i),(j)} \cap \{k\} )  \big].
\end{align*}

We relax $e_0$, $e_1$; it is clear from the preceding that for all $e_0$, $e_1$, $n$ with $G(e_0) \cap G(e_1) = \emptyset$ it holds
\begin{align}
\label{equation properties of separating E}
C^{e_0,e_1}_\emptyset = \textstyle \bigcup_{k,i}\bigcap_j   \big [ A_{e_0,e_1,i,j} \ \cup \ (B_{e_0,i} \cap C^{e_0,e_1}_{k,(i),(j)} \cap \{k\}) \big ],
\end{align}
where $A_{e_0,e_1,i,j}$ is the set of all $n$ for which $R_1(e_0,e_1,i,j,n)$ holds, \ie it is the 
$(e_0,e_1,i,j)$-section of $R_1$ (viewing the latter as a set), and similarly $B_{e_0,i}$ is the $(e_0,i)$-section of $R_2$.

We consider the functions $u_{R_1}$ and $u_{R_2}$, which are obtained from Claim \ref{claim auxiliary functions for uniformity} for $P = R_1, R_2$. Further we employ the functions $u_\vee$, $u_\wedge$ from the latter claim, the function $g$ from Claim \ref{claim local separation}, as well as the preceding $u_1$.

Following (\ref{equation properties of separating E}) we define the recursive functions $h_1,h_2,h_3,h_4: \om^5 \to \om:$
\begin{align*}
h_1(e_0,e_1,k,i,j) =& \ g(e_0,e_1,k,(i),(j))\\
h_2(e_0,e_1,k,i,j) =& \ u_\wedge(h_1(e_0,e_1,k,i,j),u_1(k))\\
h_3(e_0,e_1,k,i,j) =& \ u_\wedge(u_{R_2}(e_0,i),h_2(e_0,e_1,k,i,j))\\
h_4(e_0,e_1,k,i,j) =& \ u_\vee(u_{R_1}(e_0,e_1,i,j),h_3(e_0,e_1,k,i,j)).
\end{align*}

It is then clear from (\ref{equation properties of separating E}) and the preceding definitions that
\begin{align}
\label{equation properties of separating F}
G(e_0) \cap G(e_1) = \emptyset \ \ \Longrightarrow \ \ C^{e_0,e_1}_\emptyset = \textstyle \bigcup_{k,i}\bigcap_j \hset{h_4(e_0,e_1,k,i,j)}.
\end{align}

Next we consider some $e^\ast$ such that 
\[
\rfn{e^\ast}(e_0,e_1,\pair{\pair{k,i},j}) = h_4(e_0,e_1,k,i,j)
\]
We may assume that $\rfn{e^\ast}(e_0,e_1,\pair{t,j}) = e_\emptyset$ for all $j$ and all $t$, which are not of the form $\pair{k,i}$.

So for all $e_0$, $e_1$ with $G(e_0) \cap G(e_1) = \emptyset$ we have from (\ref{equation properties of separating F}),
\begin{align*}
C^{e_0,e_1}_\emptyset 
=& \ \textstyle \bigcup_{k,i}\bigcap_j \hset{h_4(e_0,e_1,k,i,j)}\\
=& \  \textstyle\bigcup_{k,i}\bigcap_j \hset{\rfn{e^\ast}(e_0,e_1,\pair{\pair{k,i},j})}\\
=& \ \textstyle \bigcup_{t}\bigcap_j \hset{\rfn{e^\ast}(e_0,e_1,\pair{t,j})}\\
& \hspace*{10mm} \text{(since for $t \neq\pair{i,k}$ we get the empty set)}\\
=& \ \textstyle \bigcup_{t}\bigcap_j \hset{\rfn{S(e^\ast,e_0,e_1)}(\pair{t,j})},
\end{align*}
where $S$ is as in the $s$-$m$-$n$ Theorem. 

Finally we take
\[
u: \om \times \om \to \om: u(e_0,e_1) = \pair{1,S(e^\ast,e_0,e_1)}.
\]
Clearly $u$ is recursive and total, and from the preceding analysis it is evident that $u(e_0,e_1)$ is an $\hcode$-code for $C^{e_0,e_1}_\emptyset$ when $G(e_0) \cap G(e_1) =\emptyset$. This completes the proof.

\end{document}